\documentclass[final]{amsart}
\usepackage{amssymb}
\usepackage{hhline}
\usepackage{graphicx}
\usepackage{stmaryrd}
\usepackage{float}
\usepackage{cite}
\usepackage{hyperref}

\usepackage[utf8]{inputenc}
\usepackage{listings}
\usepackage{xcolor}

\newtheorem{theorem}{Theorem}[section]

\newtheorem{proposition}[theorem]{Proposition}
\newtheorem{corollary}[theorem]{Corollary}

\theoremstyle{definition}

\newcommand{\Aut}{\mathrm{Aut\mkern 2mu}}
\newcommand{\M}{\mathrm{M\mkern 1mu}}
\newcommand{\id}{\mathrm{id\mkern 1mu}}

\title{On certain classes of commutative $g$-dimonoids}

\author{Volodymyr Gavrylkiv, Inna Hlushak, Kateryna Koporkh, Natalia Mazurenko, and Oksana Mykytsey}
\address[V.~Gavrylkiv]{Vasyl Stefanyk Carpathian National University, Ivano-Frankivsk, Ukraine} \email{vgavrylkiv@gmail.com}
\address[I.~Hlushak]{Vasyl Stefanyk Carpathian National University, Ivano-Frankivsk, Ukraine} \email{inna.hlushak@cnu.edu.ua}
\address[K.~Koporkh]{Vasyl Stefanyk Carpathian National University, Ivano-Frankivsk, Ukraine} \email{kateryna.koporkh@cnu.edu.ua}
\address[N.~Mazurenko]{Vasyl Stefanyk Carpathian National University, Ivano-Frankivsk, Ukraine} \email{nataliia.mazurenko@cnu.edu.ua}
\address[O.~Mykytsey]{Vasyl Stefanyk Carpathian National University, Ivano-Frankivsk, Ukraine} \email{oksana.mykytsei@cnu.edu.ua}


\begin{document}

\begin{abstract}
We study $g$-dimonoids, which are algebraic structures equipped with two associative binary operations satisfying a prescribed system of axioms. We develop a construction of $g$-dimonoids based on inflations of semigroups, construct several classes of $2$-dinilpotent commutative $g$-dimonoids, and investigate the isomorphism problem for these classes. As an application of these results, we obtain a structural description of all commutative $g$-dimonoids of order $3$, up to isomorphism.
\end{abstract}

\maketitle

\section*{Introduction}

Jean-Louis Loday introduced the concept of a dialgebra in~\cite{Lod} while seeking a class of (linear) algebras related to Leibniz algebras in the same way that associative algebras are related to Lie algebras. Recall that a Leibniz algebra is a linear algebra over a field whose bracket operation $[, ]$ satisfies the Leibniz identity
$$[[x, y], z] = [[x, z], y] + [x, [y, z]],$$
but is not required to be anticommutative.

Loday's approach was to distinguish the left and right multiplications by introducing two associative binary operations, denoted by $\dashv$ and $\vdash$. 
He observed that if these operations satisfy the following three axioms:
\begin{align*}
(x \dashv y) \dashv z &= x \dashv (y \vdash z), \hspace{10mm}(D_1) \\
(x \vdash y) \dashv z &= x \vdash (y \dashv z), \hspace{10mm}(D_2)\\\
(x \dashv y) \vdash z &= x \vdash (y \vdash z), \hspace{10mm}(D_3)
\end{align*}
then the bracket defined by $[x, y] = x \dashv y - y \vdash x$ satisfies the Leibniz identity.

Accordingly, Loday defined a {\em dimonoid}~\cite{Lod} as an algebraic structure $(D,\dashv,\vdash)$ consisting of a set $D$ equipped with two associative binary operations $\dashv$ and $\vdash$ satisfying axioms $(D_1)$, $(D_2)$, and $(D_3)$.  Since dialgebras are the linear analogues of dimonoids, many results on dimonoids have direct applications in the theory of dialgebras~\cite{B,F,Lod,M,ZCY}. Dimonoid theory was systematically developed in a long series of works of A.V.~Zhuchok beginning in 2008 (see, e.g., \cite{Zh2008,Zh2011AEJM,Zh2011AL}, among others) and received a significant impetus after the publication of his two monographs in 2011 \cite{Zh2011}, and 2014 \cite{Zh2014}.

T.~Pirashvili~\cite{P} introduced the concept of a duplex, a generalization of dimonoids, and constructed the free duplex.
The properties of free dimonoids were employed in~\cite{Lod} to characterize free dialgebras and to study their cohomologies. In~\cite{Liu, ZhG2017}, the notion of a dimonoid was used to define and investigate (one-sided) dirings. Furthermore, dimonoids are closely related to restrictive bisemigroups, considered by Schein~\cite{Sh}, and doppelsemigroups, introduced by A.V. Zhuchok (see, e.g., \cite{GDS2, Zh2017AU, ZhK2018}).

In~\cite{GUDS, Gdim1, Gdim2, GSDS, GLDS, Gdim3, GDSso}, several classes of dimonoids, doppelsemigroups, and extensions of doppelsemigroups were constructed and studied. Among numerous results, these works provide isomorphism criteria for the corresponding classes, structural classifications of objects of small orders, descriptions of their automorphism groups, and enumerations of pairwise nonisomorphic structures of small orders. The methods developed for studying isomorphisms in these classes are also employed in the study of the structures considered in this paper.

In~\cite{MDS}, the notion of a {\em generalized dimonoid} (or {\em $g$-dimonoid}) was introduced as an algebraic structure $(D, \dashv, \vdash)$ consisting of a set $D$ endowed with two associative binary operations $\dashv$ and $\vdash$ that satisfy the axioms $(D_1)$ and $(D_3)$. The construction of the free $g$-dimonoid was also described therein. Every semigroup $(D, \dashv)$ can naturally be regarded as a $g$-dimonoid $(D, \dashv, \dashv)$, referred to as the {\em trivial $g$-dimonoid}. In this sense, $g$-dimonoids constitute a  generalization of semigroups.

An {\em associative $0$-dialgebra}~\cite{Po}, that is, a linear space over a field equipped with two associative binary operations $\dashv$ and $\vdash$ satisfying the same axioms $(D_1)$ and $(D_3)$, can be viewed as a linear analogue of a $g$-dimonoid. Given the central role of dimonoids and $g$-dimonoids in the study of Leibniz algebras, dialgebras, and associative 0-dialgebras, their investigation by semigroup-theoretic methods constitutes a natural and promising line of research.

In~\cite{ZhYulia1}, a free $n$-nilpotent $g$-dimonoid was constructed, the least $n$-nilpotent congruence on a free $g$-dimonoid was examined, and a characterization of a free $g$-dimonoid was obtained. The construction of the free commutative $g$-dimonoid together with the description of the least commutative congruence on a free $g$-dimonoid was provided in~\cite{ZhYulia2}. Furthermore, Yu.~Zhuchok~\cite{ZhYurii_gdm2016} determined all isomorphisms between the endomorphism semigroups of free commutative $g$-dimonoids and proved that every automorphism of such an endomorphism semigroup is quasi-inner. In~\cite{MY}, Cayley-type theorems for $g$-dimonoids were established using left and right actions of sets and the concept of a dialgebra. More recently, A.~Zhuchok~\cite{ZhAEJM2022} described the least dimonoid congruence and the least semigroup congruence on the free (commutative, $n$-nilpotent) $g$-dimonoid. This year, the determinability problem for free $g$-dimonoids was solved in~\cite{Zh2026}.

In~\cite{Ggdim1}, the duality and isomorphisms of $g$-dimonoids are studied. Several new classes of $g$-dimonoids are introduced, their automorphism groups and halos are determined, and a complete classification, up to isomorphism, of all two-element $g$-dimonoids is obtained. In addition, all rectangular commutative $g$-dimonoids are completely characterized. Examples of commutative iso-dual nonabelian $g$-dimonoids and of noncommutative nonabelian rectangular $g$-dimonoids that are not dimonoids are also constructed. Finally, the numbers of all pairwise nonisomorphic $g$-dimonoids of orders up to 5, and of all pairwise nonisomorphic commutative, abelian, and rectangular $g$-dimonoids of orders up to 6, are determined by computer-assisted calculations.

This paper is organized around three main results. First, we introduce a construction of $g$-dimonoids based on inflations of semigroups. Next, we establish several classes of $2$-dinilpotent commutative $g$-dimo\-noids, derive isomorphism criteria for the $g$-dimonoids in these classes, and determine which members of these classes are dimonoids. Finally, these results are applied to obtain a structural description of all commutative $g$-dimonoids of order $3$, up to isomorphism.

\section{Preliminaries}

First, we recall several definitions and constructions concerning semigroups that will be used in the sequel.

\smallskip

An element $z$ of a semigroup $S$ is called  {\em a left zero} (resp.  {\em a right zero}) in $S$ if $za=z$ (resp. $az=z$) for any $a\in S$. An element $0$  is called  {\em a  zero} if $0$ is a left zero and a right zero.

A semigroup $S$ is called a {\em null semigroup}\cite{Howie} if there exists an element $0\in S$ such that $xy=0$ for all $x,y\in S$. In this case  $0$ is a zero of $S$.  By $O_{S^0}$ we denote a null semigroup with zero $0$ on a set $S$.  The null semigroups $O_{S^0}$ and $O_{T^z}$ are isomorphic if and only if $|S|=|T|$. If $S$ is a  set of cardinality $|S|=\kappa$, we use the notation $O_{\kappa}$ for a representative of the class of semigroups isomorphic to $O_{S^0}$. 

A semigroup $S$ with zero $0$ is called {\em nilpotent} if $S^{n+1} = \{0\}$ for some $n \in \mathbb N$. The least such $n$ is called {\em the nilpotency index} of $S$. For $n\in \mathbb N$, a nilpotent semigroup of nilpotency index $\leq n$ is called {\em $n$-nilpotent} \cite{Zh2013, ZhCA2017a}. 

Let $S$ be a semigroup and let $T$ be a subsemigroup of $S$. A surjective map $r:S\to T$ is called a \emph{retraction} if $r(a)=a$ for every $a\in T$. Equivalently, $r$ is an idempotent, that is, $r^2=r$. In this case, $T$ is called a \emph{retract} of $S$.
A semigroup $S$ is called an \emph{inflation} of its subsemigroup $T$ (see~\cite[Sec.~3.2]{CP}) if there exists a retraction $r:S\to T$ such that
$r(a)r(b)=ab$ for all $a,b\in S$. In the described situation $S$ is often referred to as an \emph{inflation of $T$ with retraction $r$}.
It is immediate from the definition that if $S$ is an inflation of $T$, then $S^2\subseteq T$.

\bigskip

We now recall the terminology and notation related to $g$-dimonoids that will be used throughout the paper.

\smallskip

Let $(D,\dashv, \vdash)$ be a $g$-dimonoid. Define new operations $\dashv^d$ and  $\vdash^d$ on $D$ by 
$$x \dashv^d y = y \vdash x\quad \text{  and  }\quad  x \vdash^d y = y \dashv x.$$
It is immediate to check that $(D,\dashv^d, \vdash^d)$ is a new $g$-dimonoid, called the {\em  dual $g$-dimonoid of $(D,\dashv, \vdash)$}, which we denote by $(D,\dashv, \vdash)^d$. It follows that the unary duality operation $d: (D, \dashv, \vdash) \mapsto (D, \dashv, \vdash)^{d}$ is involutive in the sense that $((D,\dashv, \vdash)^d)^d=(D,\dashv, \vdash)$. In fact, $(D,\dashv, \vdash)^d$ is a $g$-dimonoid if and only if $(D,\dashv, \vdash)$ is a $g$-dimonoid.  As usual, a $g$-dimonoid $(D,\dashv, \vdash)$ is said to be {\em self-dual} if $(D,\dashv, \vdash)^d=(D,\dashv, \vdash)$. 

\smallskip

A $g$-dimonoid $(D,\dashv,\vdash)$ is called {\em commutative}~\cite{ZhYulia2} if both semigroups $(D,\dashv)$ and $(D,\vdash)$ are commutative.

\smallskip

A map $\varphi : D_1 \to D_2$ is called a {\em homomorphism } from a $g$-dimonoid $(D_1,\dashv_1, \vdash_1)$ to a $g$-dimonoid $(D_2,\dashv_2, \vdash_2)$~\cite{ZhYurii_gdm2016} if $$\varphi(a\dashv_1 b)=\varphi(a)\dashv_2\varphi(b)\quad\text{  and  }\quad\varphi(a\vdash_1 b)=\varphi(a)\vdash_2\varphi(b)$$ for all $a,b\in D_1$.

A bijective homomorphism $\psi : D_1 \to D_2$ is called an {\em isomorphism } from a $g$-dimonoid $(D_1,\dashv_1, \vdash_1)$ to a $g$-dimonoid $(D_2,\dashv_2, \vdash_2)$.
If there exists an isomorphism from a $g$-dimonoid $(D_1,\dashv_1, \vdash_1)$ to a $g$-dimonoid $(D_2,\dashv_2, \vdash_2)$, then $(D_1, \dashv_1, \vdash_1)$ and $(D_2, \dashv_2, \vdash_2)$ are said to be {\em isomorphic}, denoted $(D_1,\dashv_1, \vdash_1)\cong (D_2,\dashv_2, \vdash_2)$. An isomorphism $\psi: D\to D$ is called an {\em   automorphism} of a $g$-dimonoid $(D,\dashv, \vdash)$. By $\Aut(D,\dashv, \vdash)$ we denote the automorphism group of a $g$-dimonoid $(D,\dashv, \vdash)$. 

A $g$-dimonoid $(D, \dashv, \vdash)$ is said to be  {\em iso-dual} if it is isomorphic to its dual $g$-dimonoid $(D, \dashv, \vdash)^d$.

\smallskip

The notion of an $n$-dinilpotent dimonoid was first introduced in~\cite{Zh2013} (see also \cite{ZhCA2017a}). Let us recall below this and other necessary definitions from~\cite{Zh2013}. An element $0$ of a dimonoid $(D,\dashv,\vdash)$ is called {\em zero}, if  $d * 0 = 0 = 0 * d$ for all
$d\in D$ and $*\in \{\dashv, \vdash\}$. A dimonoid $(D,\dashv,\vdash)$ with zero is called {\em dinilpotent}, if $(D,\dashv)$ and $(D, \vdash)$ are nilpotent semigroups. A dinilpotent dimonoid $(D,\dashv,\vdash)$ is called {\em $n$-dinilpotent}, if $(D,\dashv)$ and $(D, \vdash)$ are $n$-nilpotent semigroups. Replacing the word “dimonoid” with “$g$-dimonoid” in the definitions of dinilpotent and $n$-dinilpotent dimonoids, and relaxing the requirement of a single common zero $0$ to allow two (possibly distinct) zeros $0_{\dashv}$ and $0_{\vdash}$, one for each component
semigroup, one obtains the definitions of dinilpotent and $n$-dinilpotent $g$-dimonoids. We will give examples of 2-dinilpotent $g$-dimonoids in   Section~\ref{dcgdm}.

\smallskip

For a $g$-dimonoid $(D,\dashv,\vdash)$, if $\mathbb S$ and $\mathbb T$ denote the semigroups $(D,\dashv)$ and $(D,\vdash)$, respectively, then $\mathbb S \rbag \mathbb T$ stands for the $g$-dimonoid $(D,\dashv,\vdash)$.

\section{Inflations and $g$-dimonoids}\label{inf_gdm}

In this section, we present a method for constructing $g$-dimonoids based on the notion of inflations of semigroups.

\smallskip

Let $r:S\to T$ be a surjective map. A subset $P\subseteq S$ is called a {\em transversal} of $r$ if $P$ contains exactly one representative from each fiber $r^{-1}(t)$, $t\in T$.

\begin{proposition}\label{prop:new_inflation}
Let $(S,\cdot)$ be an inflation of a semigroup $T$ with retraction $r:S\to T$, and let $P=\{p_t\mid t\in T\}\subseteq S$ be a transversal
of $r$. Define a binary operation $\star$ on $S$ by $$x\star y=p_{r(x)r(y)},\qquad x,y\in S.$$
Then $(S,\star)$ is a semigroup, and $P$ is a subsemigroup of $(S,\star)$. Moreover, the semigroup $(S,\star)$ is an inflation of $P$ with retraction
$$
\rho:S\to P,\qquad\rho(x)=p_{r(x)},
$$
and there exists an involutive isomorphism from $(S,\cdot)$ onto $(S,\star)$.
\end{proposition}

\begin{proof}
We first prove that the operation $\star$ is associative. Let $x,y,z\in S$. Then $(x\star y)\star z = p_{\,r(x\star y)r(z)}$.
Since  $r(p_t)=t$ for each $t\in T$, we have 
$$r(x\star y) = r(p_{r(x)r(y)}) =r(x)r(y).$$
Hence $(x\star y)\star z = p_{(r(x)r(y))r(z)}$. Similarly, $x\star(y\star z) = p_{r(x)(r(y)r(z))}$.

Since multiplication in $T$ is associative, it follows that
$$(x\star y)\star z = p_{(r(x)r(y))r(z)} = p_{r(x)(r(y)r(z))} =x\star(y\star z).$$
Consequently, $(S,\star)$ is a semigroup.

Now let $p_t,p_u\in P$. Since $r(p_t)=t$ and $r(p_u)=u$,
$$p_t\star p_u = p_{r(p_t)r(p_u)} = p_{tu}\in P.$$
Hence $P$ is closed under $\star$, and therefore is a subsemigroup of
$(S,\star)$.

The map $\rho:S\to P$ is surjective by definition of $P$. Moreover, for every $x\in S$,
$$\rho^2(x) = \rho(\rho(x)) = \rho(p_{r(x)}) = p_{r(p_{r(x)})} = p_{r(x)} = \rho(x).$$

Next, for all $x,y\in S$,
$$\rho(x)\star\rho(y) = p_{r(x)}\star p_{r(y)} = p_{\,r(p_{r(x)})\,r(p_{r(y)})} = p_{r(x)r(y)} = x\star y.$$

Thus $\rho$ is a surjective idempotent map satisfying $\rho(x)\star\rho(y)=x\star y$ for all $x,y\in S$. Therefore, $(S,\star)$ is an inflation of the
subsemigroup $P$.

To complete the proof, we construct an involutive isomorphism from $(S,\cdot)$ onto $(S,\star)$.

Define a map $\psi:S\to S$ by $\psi(t)=p_t$, $\psi(p_t)=t$ for every $t\in T$, and
$\psi(x)=x$ for all remaining elements of $S$. Clearly, $\psi$ is an involutive bijection in the sense that $\psi^2=\id_S$ and
$r(\psi(x))=r(x)$ for every $x\in S$.

Now let $x,y\in S$. Since $(S,\cdot)$ is an inflation of $T$, it follows that $x\cdot y=r(x)r(y)\in T$.
Hence
$$
\psi(x\cdot y)
 =p_{r(x)r(y)}
 =p_{r(\psi(x))\,r(\psi(y))}
 =\psi(x)\star\psi(y).
$$
Therefore, $\psi$ is an involutive isomorphism from $(S,\cdot)$ onto $(S,\star)$.
\end{proof}

\bigskip

We now pass to $g$-dimonoids.

\begin{theorem}\label{thm:inflation_g_dimonoid}
Let $(S,\cdot)$ be an inflation of a semigroup $T$ with retraction $r:S\to T$, and let $P=\{p_t\mid t\in T\}$ be an arbitrary transversal of $r$.
Define binary operations on $S$ by
$$x\dashv y:=xy,\qquad x\vdash y:=p_{r(x)r(y)}.$$
Then $(S,\dashv,\vdash)$ is a $g$-dimonoid whose component semigroups
$(S,\dashv)$ and $(S,\vdash)$ are isomorphic.
Moreover, $(S,\dashv,\vdash)$ is iso-dual whenever $(S,\dashv)$ is commutative, and it is a dimonoid if and only if $p_t=t$  for all $t\in T^3$.
\end{theorem}

\begin{proof}
Since $(S,\dashv)$ coincides with the original semigroup $(S,\cdot)$, the ope\-ration $\dashv$ is associative. By Proposition~\ref{prop:new_inflation},
the operation $\vdash$ is associative as well.

It remains to verify axioms $(D_1)$ and $(D_3)$.
For arbitrary $x,y,z\in S$, we have $(x\dashv y)\dashv z=(xy)z$. Since $S$ is an inflation of $T$, it follows that $r(xy)=r(r(x)r(y))=r(x)r(y)$,
and therefore
$$(xy)z = r(xy)r(z) =(r(x)r(y))r(z).$$

On the other hand, $x\dashv(y\vdash z)=x\,p_{r(y)r(z)}$.
Since $r(p_t)=t$ for every $t\in T$, we obtain
$$ x\,p_{r(y)r(z)} = r(x)\,r(p_{r(y)r(z)}) = r(x)(r(y)r(z)).$$
Since multiplication in $T$ is associative,
$(r(x)r(y))r(z) = r(x)(r(y)r(z))$.
Hence $(x\dashv y)\dashv z = x\dashv(y\vdash z)$, and axiom $(D_1)$ holds.

Next,
$$(x\dashv y)\vdash z = (xy)\vdash z = p_{r(xy)r(z)}= p_{(r(x)r(y))r(z)}.$$

Furthermore,
$$x\vdash(y\vdash z) = x\vdash p_{r(y)r(z)} = p_{r(x)\,r(p_{r(y)r(z)})} = p_{r(x)(r(y)r(z))}.$$

Again, associativity of multiplication in $T$ yields $$(x\dashv y)\vdash z = x\vdash(y\vdash z).$$ Thus axiom $(D_3)$ also holds.
Therefore, $(S,\dashv,\vdash)$ is a $g$-dimonoid whose component semigroups $(S,\dashv)$ and $(S,\vdash)$ are isomorphic by Proposition~\ref{prop:new_inflation}.

\smallskip

Next, assume  that $(S,\dashv)$ is commutative. By Proposition~\ref{prop:new_inflation}, there exists an involutive isomorphism $\psi:S\to S$ from  $(S,\dashv)$ onto $(S,\vdash)$. Hence $(S,\vdash)$ is also commutative.
Therefore,
$$
\psi(x\dashv y)=\psi(x)\vdash\psi(y)
=\psi(y)\vdash\psi(x)
=\psi(x)\dashv^d\psi(y).
$$

Taking into account that $\psi$ is involutive in the sense that $\psi^2=\id_S$, we conclude that
$$
x\vdash y =\psi(\psi(x))\vdash\psi(\psi(y))
=\psi(\psi(x)\dashv\psi(y)).
$$
Applying $\psi$ to both sides yields
$$
\psi(x\vdash y)
=\psi(x)\dashv\psi(y)
=\psi(y)\dashv\psi(x)
=\psi(x)\vdash^d\psi(y).
$$

Therefore, $\psi$ is an isomorphism from $(S,\dashv,\vdash)$ onto its dual $g$-dimo\-noid $(S,\dashv,\vdash)^d$, and hence $(S,\dashv,\vdash)$ is iso-dual.

\smallskip

Finally, we determine when axiom $(D_2)$ holds. We have
$$(x\vdash y)\dashv z = p_{r(x)r(y)}\,z = r(p_{r(x)r(y)})\,r(z) = (r(x)r(y))r(z),$$
where the second equality follows from the defining property of an
inflation. On the other hand,
$$x\vdash(y\dashv z) = p_{r(x)r(y\dashv z)} = p_{r(x)r(yz)} = p_{r(x)(r(y)r(z))} = p_{(r(x)r(y))r(z)}.$$

Therefore, axiom $(D_2)$ is equivalent to $t=p_t$ for every element $t=(r(x)r(y))r(z)$. Since $r$ is surjective, this is equivalent to
$p_t=t$ for all $t\in T^3$. Hence $(S,\dashv,\vdash)$ is a dimonoid if and only if $p_t=t$ for all $t\in T^3$.
\end{proof}

The next proposition provides a criterion for two $g$-dimonoids arising from different transversals to be isomorphic.

\begin{proposition}\label{thm:isomorphic_transversals}
Let $P=\{p_t\mid t\in T\}$ and
$Q=\{q_t\mid t\in T\}$ be two transversals of an inflation $(S,\cdot)$ of a semigroup $T$ with retraction  $r:S\to T$.
Let $(S,\dashv,\vdash_P)$ and $(S,\dashv,\vdash_Q)$ be the $g$-dimonoids obtained from $P$ and $Q$, respectively, by the construction of Theorem~\ref{thm:inflation_g_dimonoid}.
Then $(S,\dashv,\vdash_P)\cong(S,\dashv,\vdash_Q)$ if and only if there exists an automorphism
$\psi$ of the semigroup $(S,\cdot)$ such that
$$\psi(p_{r(x)r(y)})=q_{\,r(\psi(x))\,r(\psi(y))}$$
for all $x,y\in S$.
\end{proposition}

\begin{proof}
Suppose that $\psi:S\to S$ is an isomorphism from $(S,\dashv,\vdash_P)$ to $(S,\dashv,\vdash_Q)$. Since $x\dashv y=xy$ in both $(S,\dashv,\vdash_P)$ and $(S,\dashv,\vdash_Q)$, it follows that $\psi$ is an automorphism of the semigroup $(S,\cdot)$.

Furthermore,
$$\psi(x\vdash_Py) = \psi(x)\vdash_Q\psi(y)$$
for all $x,y\in S$. By the definitions of $\vdash_P$ and $\vdash_Q$,
this is equivalent to
$$\psi(p_{r(x)r(y)}) = q_{\,r(\psi(x))\,r(\psi(y))}$$
for all $x,y\in S$.

Conversely, suppose that $\psi$ is an automorphism of the semigroup
$(S,\cdot)$ satisfying
$$\psi(p_{r(x)r(y)}) = q_{\,r(\psi(x))\,r(\psi(y))}$$
for all $x,y\in S$.

Since $\psi$ is an automorphism of $(S,\cdot)$ and
$$\psi(x\vdash_Py) = \psi(p_{r(x)r(y)}) = q_{\,r(\psi(x))\,r(\psi(y))} = \psi(x)\vdash_Q\psi(y),$$
it follows that $\psi$ is an isomorphism of $g$-dimonoids $(S,\dashv,\vdash_P)$ and  $(S,\dashv,\vdash_Q)$. This completes the proof.
\end{proof}

\begin{corollary}\label{cor:isomorphic_transversals}
Under the assumptions of Proposition~\ref{thm:isomorphic_transversals},
assume in addition that every automorphism $\psi$ of the semigroup
$(S,\cdot)$ commutes with the retraction $r$, that is, $r\circ\psi=\psi\circ r$.
Then the $g$-dimonoids $(S,\dashv,\vdash_P)$ and $(S,\dashv,\vdash_Q)$ are isomorphic if and only if
there exists an automorphism $\psi$ of the semigroup $(S,\cdot)$
such that $\psi(p_t)=q_{\psi(t)}$ for every $t\in T^2$.
\end{corollary}

\begin{proof}
By Proposition~\ref{thm:isomorphic_transversals}, the $g$-dimonoids
$(S,\dashv,\vdash_P)$ and $(S,\dashv,\vdash_Q)$ are isomorphic if and only if there exists an
automorphism $\psi$ of $(S,\cdot)$ satisfying
$$\psi(p_{r(x)r(y)}) = q_{\,r(\psi(x))\,r(\psi(y))}$$
for all $x,y\in S$.

Since $\psi$ commutes with the retraction $r$ and $\psi$ is a semigroup automorphism,
$$r(\psi(x))\,r(\psi(y)) = \psi(r(x))\psi(r(y)) = \psi(r(x)r(y)),$$
and hence
$$\psi(p_{r(x)r(y)}) = q_{\psi(r(x)r(y))}$$
for all $x,y\in S$.

Since every element of $T^2$ can be written as $r(x)r(y)$ for suitable
$x,y\in S$, the latter condition is equivalent to
$\psi(p_t)=q_{\psi(t)}$
for every $t\in T^2$.
\end{proof}

\section{Some classes of $2$-dinilpotent commutative $g$-dimonoids}\label{dcgdm}

In this section, we construct several classes of $2$-dinilpotent commutative $g$-dimonoids and investigate the isomorphism problem for $g$-dimonoids in these classes.

\begin{proposition}\label{comm_mon_gdm}
Let $D$ be a set with $|D|\geq 3$, and let $a,b,0_{\dashv}\in D$ and $a,c,0_{\vdash}\in D$ be pairwise distinct within each triple.
Define two binary operations $\dashv$ and $\vdash$ on $D$ by
$$
x\dashv y=
\begin{cases}
b,&\text{if }x=y=a\\
0_{\dashv},&\text{otherwise}
\end{cases}
\quad\text{and}\quad
x\vdash y=
\begin{cases}
c,&\text{if }x=y=a\\
0_{\vdash},&\text{otherwise}.
\end{cases}
$$
Then $(D,\dashv,\vdash)$ is a  $2$-dinilpotent commutative $g$-dimonoid. Moreover, it is a dimonoid if and only if
$0_{\dashv}=0_{\vdash}$.
\end{proposition}

\begin{proof}
Both operations are clearly commutative.

Since $b\neq a$ and $0_{\dashv}\neq a$, we have
$$(x\dashv y)\dashv z = 0_{\dashv} = x\dashv(y\dashv z),$$
for all $x,y,z\in D$. Hence $\dashv$ is associative.

Similarly, since $c\neq a$ and $0_{\vdash}\neq a$, we obtain
$$(x\vdash y)\vdash z = 0_{\vdash} = x\vdash(y\vdash z),$$
for all $x,y,z\in D$. Thus $\vdash$ is associative.

It follows that every triple product in $(D,\dashv)$ equals $0_{\dashv}$, and every triple product in $(D,\vdash)$ equals $0_{\vdash}$. Hence both component semigroups are $2$-nilpotent, and therefore $(D,\dashv,\vdash)$ is a $2$-dinilpotent $g$-dimonoid.

Furthermore, since $x\dashv y\in\{b,0_{\dashv}\},\quad y\vdash z\in\{c,0_{\vdash}\}$, and none of the elements
$b$, $0_{\dashv}$, $c$, and $0_{\vdash}$ equals $a$, we obtain
$$
(x\dashv y)\dashv z = 0_{\dashv} = x\dashv(y\vdash z)
\quad\text{ and }\quad
(x\dashv y)\vdash z = 0_{\vdash} = x\vdash(y\vdash z).
$$
Hence axioms $(D_1)$ and $(D_3)$ hold.

Finally, since $x\vdash y\in\{c,0_{\vdash}\}$ and neither $c$ nor $0_{\vdash}$ equals $a$, it follows that $(x\vdash y)\dashv z = 0_{\dashv}$
for all $x,y,z\in D$. On the other hand,  $y\dashv z\in\{b,0_{\dashv}\}$ and neither $b$ nor $0_{\dashv}$ equals $a$, whence
$x\vdash(y\dashv z) = 0_{\vdash}$ for all $x,y,z\in D$. Therefore, axiom $(D_2)$ holds if and only if
$0_{\dashv}=0_{\vdash}$. Consequently, $(D,\dashv,\vdash)$ is a dimonoid if and only if $0_{\dashv}=0_{\vdash}$.
\end{proof}

The following proposition provides a criterion for the isomorphism of $g$-dimonoids defined above.

\begin{proposition}\label{iso_comm_mon_gdm}
For each $i\in\{1,2\}$, let $D_i$ be a set with $|D_i|\ge3$,  let $a_i,b_i,0_{\dashv,i}\in D_i$ and
$a_i,c_i,0_{\vdash,i}\in D_i$ be pairwise distinct within each triple, and let $(D_i,\dashv_i,\vdash_i)$ be the $g$-dimonoid defined by
$$
x\dashv_i y=
\begin{cases}
b_i,&\text{if }x=y=a_i\\
0_{\dashv,i},&\text{otherwise}
\end{cases}
\quad\text{and}\quad
x\vdash_i y=
\begin{cases}
c_i,&\text{if }x=y=a_i\\
0_{\vdash,i},&\text{otherwise}.
\end{cases}
$$
Then $(D_1,\dashv_1,\vdash_1)\cong(D_2,\dashv_2,\vdash_2)$ if and only if there exists a bijection
$\psi:D_1\to D_2$ such that
$$
\psi(a_1)=a_2,\
\psi(b_1)=b_2,\
\psi(c_1)=c_2,\
\psi(0_{\dashv,1})=0_{\dashv,2},\
\psi(0_{\vdash,1})=0_{\vdash,2}.
$$
\end{proposition}

\begin{proof}
Suppose that $\psi:D_1\to D_2$ is an isomorphism from $(D_1,\dashv_1,\vdash_1)$ to $ (D_2,\dashv_2,\vdash_2)$.
Then $\psi$ is simultaneously an isomorphism of the semigroups $(D_1,\dashv_1)$ and $(D_2,\dashv_2)$, and of
$(D_1,\vdash_1)$ and $(D_2,\vdash_2)$. Hence it preserves their zeros, that is,
$\psi(0_{\dashv,1})=0_{\dashv,2}$ and $\psi(0_{\vdash,1})=0_{\vdash,2}$.

Since $a_1\dashv_1a_1=b_1$, we obtain $\psi(a_1)\dashv_2\psi(a_1)=\psi(b_1)$.
Taking into account that $\psi(0_{\dashv,1})=0_{\dashv,2}$ and $b_1\neq0_{\dashv,1}$, we conclude that $\psi(b_1)\neq0_{\dashv,2}$.
Hence $\psi(a_1)=a_2$, since $a_2$ is the unique element $x\in D_2$ such that $x\dashv_2x\neq0_{\dashv,2}$.
Consequently, $\psi(b_1)=b_2$.
Similarly, $\psi(c_1)=c_2$.

Conversely, let $\psi:D_1\to D_2$ be a bijection satisfying
$$
\psi(a_1)=a_2,\
\psi(b_1)=b_2,\
\psi(c_1)=c_2,\
\psi(0_{\dashv,1})=0_{\dashv,2},\
\psi(0_{\vdash,1})=0_{\vdash,2}.
$$
If $(x,y)=(a_1,a_1)$, then
$$ \psi(x\dashv_1y) = \psi(b_1) = b_2 = a_2\dashv_2a_2 = \psi(x)\dashv_2\psi(y).$$
If $(x,y)\neq(a_1,a_1)$, then
$$\psi(x\dashv_1y) = \psi(0_{\dashv,1}) = 0_{\dashv,2} = \psi(x)\dashv_2\psi(y).$$
Hence $\psi$ preserves $\dashv$. The verification that $\psi$ preserves $\vdash$ is analogous. Therefore, $\psi$ is an isomorphism.
\end{proof}

\begin{corollary}\label{isodual_comm_mon_gdm}
Let $(D,\dashv,\vdash)$ be the $g$-dimonoid defined in
Proposition~\ref{comm_mon_gdm}. Then $(D,\dashv,\vdash)$ is iso-dual if
and only if $$b=0_{\vdash}\Longleftrightarrow c=0_{\dashv}.$$
\end{corollary}

\begin{proof}
Since $(D,\dashv,\vdash)$ is commutative, its dual $g$-dimonoid
$(D,\dashv,\vdash)^d$ is given by
$$
x\dashv^d y=x\vdash y,\qquad
x\vdash^d y=x\dashv y.
$$
Hence
$$
x\dashv^d y=
\begin{cases}
c,&\text{if }x=y=a\\
0_{\vdash},&\text{otherwise}
\end{cases}
\quad\text{and}\quad
x\vdash^d y=
\begin{cases}
b,&\text{if }x=y=a\\
0_{\dashv},&\text{otherwise}.
\end{cases}
$$
Applying Proposition~\ref{iso_comm_mon_gdm} to
$(D,\dashv,\vdash)$ and $(D,\dashv,\vdash)^d$, we conclude that
$(D,\dashv,\vdash)$ is iso-dual if and only if there exists a bijection
$\psi:D\to D$ satisfying
$$
\psi(a)=a,\quad
\psi(b)=c,\quad
\psi(c)=b,\quad
\psi(0_{\dashv})=0_{\vdash},\quad
\psi(0_{\vdash})=0_{\dashv}.
$$

Since $a,b,0_{\dashv}$ are pairwise distinct, and $a,c,0_{\vdash}$ are
pairwise distinct, we have
$$
a\neq b,\quad
a\neq c,\quad
a\neq0_{\dashv},\quad
a\neq0_{\vdash},\quad
b\neq0_{\dashv},\quad
c\neq0_{\vdash}.
$$
Moreover, the equalities $b=c$ and $0_{\dashv}=0_{\vdash}$ do not cause any inconsistency, since in either case the prescribed images coincide. Thus the only possible incompatibility occurs when exactly one of the equalities $b=0_{\vdash}$ and $c=0_{\dashv}$ holds. Therefore, the above conditions define a bijection on the distinguished elements if and only if 
$$b=0_{\vdash}\Longleftrightarrow c=0_{\dashv}.$$ 
Such a bijection extends arbitrarily to a bijection of $D$, completing the proof.
\end{proof}

The following proposition was proved in~\cite{Ggdim1}.

\begin{proposition}\label{null_dm}
Let $(D, \dashv)$ be a null semigroup with zero $0$ and $(D, \vdash)$ be an arbitrary semigroup. An algebraic structure $(D,\dashv, \vdash)$ is a $g$-dimonoid if and only if $x\vdash y\vdash z = 0 \vdash z$ for all $x,y,z\in D$.
\end{proposition}

\begin{proposition}\label{comm_nullmon_gdm}
Let $D$ be a set with $|D|\geq 3$, and let $a,c,0_{\vdash}\in D$ be pairwise distinct.
Let $(D, \dashv)$ be a null semigroup with zero $0_{\dashv}$, where $0_{\dashv}\neq a$, and define the binary operation $\vdash$ on $D$ by
$$
x\vdash y=
\begin{cases}
c,&\text{if }x=y=a\\
0_{\vdash},&\text{otherwise}.
\end{cases}
$$
Then $(D,\dashv,\vdash)$ is a $2$-dinilpotent commutative $g$-dimonoid. Moreover, it is a dimonoid if and only if $0_{\dashv}=0_{\vdash}$.
\end{proposition}

\begin{proof}
The operation $\vdash$ is clearly associative and commutative. Since
$c\neq a$ and $0_{\dashv}\neq a$, we have
$$(x\vdash y)\vdash z = 0_{\vdash} = 0_{\dashv}\vdash z$$
for all $x,y,z\in D$. Hence, by Proposition~\ref{null_dm},
$(D,\dashv,\vdash)$ is a commutative $g$-dimonoid.

Moreover, $(D,\dashv)$ is $1$-nilpotent, and $(D,\vdash)$ is $2$-nilpotent since every triple product in $(D,\vdash)$ equals $0_{\vdash}$. Therefore, $(D,\dashv,\vdash)$ is a $2$-dinilpotent commutative $g$-dimonoid.

Finally, since $(D,\dashv)$ is a null semigroup,
$$(x\vdash y)\dashv z = 0_{\dashv}$$
for all $x,y,z\in D$. Moreover,
$$x\vdash(y\dashv z) = x\vdash0_{\dashv} = 0_{\vdash},$$
because $0_{\dashv}\neq a$. Therefore, axiom $(D_2)$ holds if and only if $0_{\dashv}=0_{\vdash}$. Consequently, $(D,\dashv,\vdash)$ is a dimonoid if and only if
$0_{\dashv}=0_{\vdash}$.
\end{proof}

The proof of the following proposition is analogous to that of Proposition~\ref{iso_comm_mon_gdm} and is therefore omitted.

\begin{proposition}\label{iso_comm_nullmon_gdm}
For each $i\in\{1,2\}$, let $D_i$ be a set with $|D_i|\ge3$, let
$a_i,c_i,0_{\vdash,i}\in D_i$ be pairwise distinct, let
$(D_i,\dashv_i)$ be a null semigroup with zero $0_{\dashv,i}\neq a_i$, and
define the binary operation $\vdash_i$ on $D_i$ by
$$
x\vdash_i y=
\begin{cases}
c_i,&\text{if }x=y=a_i\\
0_{\vdash,i},&\text{otherwise}.
\end{cases}
$$
Then $(D_1,\dashv_1,\vdash_1)\cong(D_2,\dashv_2,\vdash_2)$ if and only if there exists a bijection $\psi:D_1\to D_2$ such that
$$
\psi(a_1)=a_2,\quad
\psi(c_1)=c_2,\quad
\psi(0_{\dashv,1})=0_{\dashv,2},\quad
\psi(0_{\vdash,1})=0_{\vdash,2}.
$$
\end{proposition}

\section{Commutative $g$-dimonoids of order 3}

In this section, we focus on describing, up to isomorphism, all commutative $g$-dimonoids of order $3$.

\smallskip

Among the $19683$ possible binary operations on a three-element set $S$, precisely $113$ are associative. In other words, there exist exactly $113$ distinct three-element semigroups. However, many of these semigroups are isomorphic, and as a result, there are essentially only $24$ pairwise nonisomorphic semigroups of order $3$, see \cite{GR1, G9}.
Among these $24$ pairwise nonisomorphic semigroups of order $3$, there are $12$ commutative semigroups. 

Lists of all pairwise nonisomorphic commutative semigroups of order $3$ and their automorphism groups are presented in Table~\ref{tab:auts3} taken from~\cite{G9}. In this table, the notation $S^{+1}$ denotes a monoid obtained from $S$ by adjoining the extra identity $1$ (whether or not $S$ is a monoid), and $S^{+0}$ denotes a semigroup obtained from $S$ by adjoining the extra zero $0$ (regardless of whether $S$ has a zero). The notation $M^{\tilde{1}}$ stands for the semigroup obtained from a monoid $M$ with  identity $1$ by adjoining an extra element $\tilde{1}$ by putting $\tilde{1}* m=m* \tilde{1}=m$ for all $m\in M$ and ${\tilde{1}}*{\tilde{1}}=1$. The notation $\M_{r,m}$ refers to a finite monogenic semigroup of index $r$ and period $m$, $L_n$ denotes the  linear semilattice $\{0, 1,\ldots, n\!-\!1\}$ of order $n$ endowed with the operation of minimum,  $C_n$ stands for the cyclic group of $n$-th roots of $1$, and $O_n^m$ denotes the semigroup of order $n$ with zero in which exactly $m$ nonzero elements are idempotents, and all other products are equal to zero. More details about these semigroups can be found in \cite{GR1}.

\begin{table}[H]
\centering
\resizebox{12cm}{!}{
\begin{tabular}{|c|c|c|c|c|c|c|c|c|c|c|c|c|}
        \hline
        $S$ & $C_3$ & $O_3$  & $\M_{2,2}$ & $C_2^{+1}$ &  $C_2^{\tilde{1}}$ & $\M_{3,1}$&  $O_2^{+1}$ & $O_2^{+0}$ & $L_3$ & $C_2^{+0}$    & $O_3^2$  & $O_3^1$ \\
        \hline
        $\Aut(S)$ &  $C_2$ & $C_2$ & $C_1$ & $C_1$ & $C_{1}$  & $C_{1}$ & $C_{1}$ & $C_{1}$ & $C_{1}$ & $C_{1}$ & $C_{2}$ & $C_{1}$  \\
        \hline
\end{tabular}
}
\smallskip
\caption{Nonisomorphic commutative $3$-element semigroups and their automorphism groups}\label{tab:auts3}
\end{table}

The following theorem,  established in~\cite{Gdim2}, provides a complete classification of all pairwise nonisomorphic commutative dimonoids of order~$3$.

\begin{theorem}\label{num_comm_dm}
Up to isomorphism, there exist $14$ three-element commutative dimonoids: $12$ trivial dimonoids and a pair of mutually dual nontrivial dimonoids, namely $\M_{3,1} \rbag O_3$ and $O_3 \rbag \M_{3,1}$.
\end{theorem}

In~\cite{Ggdim1}, computer-assisted calculations established that there are $22$ pairwise nonisomorphic commutative $g$-dimonoids of order $3$. In view of Theorem~\ref{num_comm_dm}, it follows that exactly $8$ pairwise nonisomorphic commutative $g$-dimonoids of order $3$ are not dimonoids. In what follows, we construct all eight of these $g$-dimonoids as an application of the results established in Sections~\ref{inf_gdm} and~\ref{dcgdm}.

\bigskip

The first two commutative $g$-dimonoids are described using the following theorem, established in~\cite{Ggdim1}.

\begin{theorem}\label{iso-dual_O}
Let $S$ be a set, and $0\in S$, $z\in S$. An algebraic structure $O_{S^0} \rbag O_{S^z} = (S, \dashv, \vdash)$, where $(S, \dashv)$ and $(S, \vdash)$ are null semigroups with zeros 
$0$ and $z$, respectively, is an iso-dual commutative $g$-dimonoid. Moreover, if $0\ne z$, the $g$-dimonoid $O_{S^0} \rbag O_{S^z}$ is not a dimonoid.
\end{theorem}

\noindent{\bf Commutative $g$-dimonoid \#1.} Let $S=\{0,1,2\}$. According to  Theorem~\ref{iso-dual_O}, an algebraic structure $O_{S^0} \rbag O_{S^1} = (S, \dashv, \vdash)$, where $(S, \dashv)$ and $(S, \vdash)$ are null semigroups with zeros $0$ and $1$, respectively, is an iso-dual commutative $g$-dimonoid of order $3$, which is not a dimonoid.

\bigskip

\noindent{\bf Commutative $g$-dimonoid \#2.} Let $S=\{0,1,2\}$, and let $D=\{1,2\}$. If $(D,\dashv, \vdash)$ is a  $g$-dimonoid, then the algebraic structure $(D,\dashv, \vdash)^{+0}$  obtained from $D$ by adjoining an extra zero $0\notin D$, defined by 
$$0\dashv d=d\dashv 0=0=0\vdash d=d\vdash 0\quad\text{ for all }d\in D\cup \{0\},$$
 is a $g$-dimonoid as well, see~\cite{Ggdim1}.  By Theorem~\ref{iso-dual_O}, the algebraic structure $O_{D^1}\rbag O_{D^2}=(D,\dashv,\vdash)$, where $(D,\dashv)$ and
$(D,\vdash)$ are null semigroups with zeros $1$ and $2$, respectively, is an iso-dual commutative $g$-dimonoid of order $2$ that is not a dimonoid. Consequently, adjoining the extra zero $0$ yields the three-element algebraic structure $(O_{D^1}\rbag O_{D^2})^{+0}$, which is also an iso-dual commutative $g$-dimonoid but not a dimonoid.

\bigskip

The next three $g$-dimonoids are obtained by applying the inflation construction developed in Section~\ref{inf_gdm}.

\medskip

\noindent{\bf Commutative $g$-dimonoid \#3.} Let $S=\{0,1,2\}$ be the semigroup given by the following Cayley table:
$$
\begin{array}{c|ccc}
\cdot & 0 & 1 & 2\\ \hline
0 & 0 & 0 & 2\\
1 & 0 & 0 & 2\\
2 & 2 & 2 & 0
\end{array}
$$
which is an inflation of the cyclic subgroup $T=\{0,2\}$ with retraction
$$r(0)=r(1)=0,\qquad r(2)=2.$$
The semigroup $(S, \cdot)$ is isomorphic to  $C_2^{\tilde{1}}$ from Table~\ref{tab:auts3}. 

Choose the transversal $P=\{1,2\}$, that is, $p_0=1$ and $p_2=2$.
By Theorem~\ref{thm:inflation_g_dimonoid}, the operations
$$
x\dashv y:=xy,\qquad
x\vdash y:=p_{r(x)r(y)}
$$
yield the $g$-dimonoid given by the following Cayley tables:
$$
\begin{array}{c|ccc}
\dashv & 0 & 1 & 2\\ \hline
0 & 0 & 0 & 2\\
1 & 0 & 0 & 2\\
2 & 2 & 2 & 0
\end{array}
\qquad\qquad
\begin{array}{c|ccc}
\vdash & 0 & 1 & 2\\ \hline
0 & 1 & 1 & 2\\
1 & 1 & 1 & 2\\
2 & 2 & 2 & 1
\end{array}
$$

Since $T^3=T$ and $p_0\neq 0$, Theorem~\ref{thm:inflation_g_dimonoid} implies that this $g$-dimonoid is not a dimonoid. On the other hand, as $(S,\dashv)$ is commutative, the same theorem shows that it is iso-dual.

Consider also the canonical transversal $Q=\{0,2\}$, where $q_0=0$ and $q_2=2$. By Theorem~\ref{thm:inflation_g_dimonoid}, this transversal yields the trivial dimonoid $S$. The semigroup $(S,\cdot)$ has only the identity automorphism (see~Table~\ref{tab:auts3}), which clearly commutes with the retraction $r$. Therefore, we may apply Corollary~\ref{cor:isomorphic_transversals}. 
However, the identity automorphism $\psi$ does not satisfy the condition.
Indeed, 
$$
\psi(p_0)= p_0 = 1\neq 0 = q_0=q_{\psi(0)}.
$$

Hence, by Corollary~\ref{cor:isomorphic_transversals}, the $g$-dimonoids determined by the transversals $P=\{1,2\}$ and $Q=\{0,2\}$ are not isomorphic.

\bigskip

\noindent{\bf Commutative $g$-dimonoid \#4.} Let $S=\{0,1,2\}$ be the semigroup given by the following Cayley table:
$$
\begin{array}{c|ccc}
\cdot & 0 & 1 & 2\\ \hline
0 & 0 & 1 & 1\\
1 & 1 & 0 & 0\\
2 & 1 & 0 & 0
\end{array}
$$
which is an inflation of the cyclic subgroup $T=\{0,1\}$ with retraction
$$r(0)=0,\qquad r(1)=r(2)=1.$$ 
The semigroup $(S, \cdot)$ is monogenic of index 2 and period 2 generated by the element 2.

Choose the transversal $P=\{0,2\}$, that is, $p_0=0$ and $p_1=2$.
By Theorem~\ref{thm:inflation_g_dimonoid}, the operations
$$
x\dashv y:=xy,\qquad
x\vdash y:=p_{r(x)r(y)}
$$
yield the $g$-dimonoid given by the following Cayley tables:
$$
\begin{array}{c|ccc}
\dashv & 0 & 1 & 2\\ \hline
0 & 0 & 1 & 1\\
1 & 1 & 0 & 0\\
2 & 1 & 0 & 0
\end{array}
\qquad\qquad
\begin{array}{c|ccc}
\vdash & 0 & 1 & 2\\ \hline
0 & 0 & 2 & 2\\
1 & 2 & 0 & 0\\
2 & 2 & 0 & 0
\end{array}
$$

Since $T^3=T$ and $p_1\neq 1$, Theorem~\ref{thm:inflation_g_dimonoid} implies that this $g$-dimonoid is not a dimonoid. On the other hand, as $(S,\dashv)$ is commutative, the same theorem shows that it is iso-dual.

Consider also the canonical transversal $Q=\{0,1\}$, where $q_0=0$ and $q_1=1$.  By Theorem~\ref{thm:inflation_g_dimonoid}, this transversal yields the trivial dimonoid $S$.
The semigroup $(S,\cdot)$ has only the identity automorphism (see~Table~\ref{tab:auts3}), which clearly commutes with the retraction $r$. Therefore, we may apply Corollary~\ref{cor:isomorphic_transversals}.
However, the identity automorphism $\psi$ does not satisfy the condition.
Indeed, 
$$
\psi(p_1)= p_1 = 2\neq 1 =q_1 =q_{\psi(1)}.
$$

Hence, by Corollary~\ref{cor:isomorphic_transversals}, the $g$-dimonoids determined by the transversals $P=\{0,2\}$ and $Q=\{0,1\}$ are not isomorphic.

\bigskip

\noindent{\bf Commutative $g$-dimonoid \#5.} Let $S=\{0,1,2\}$ be the semigroup given by the following Cayley table:
\[
\begin{array}{c|ccc}
\cdot & 0 & 1 & 2\\ \hline
0 & 0 & 0 & 0\\
1 & 0 & 0 & 0\\
2 & 0 & 0 & 2
\end{array}
\]
which is an inflation of the subsemigroup $T=\{0,2\}$ with retraction
$$r(0)=r(1)=0,\qquad r(2)=2.$$
The semigroup $(S, \cdot)$ is isomorphic to  $O_3^1$ from Table~\ref{tab:auts3}.

Choose the transversal $P=\{1,2\}$, that is, $p_0=1$ and $p_2=2$.
By Theorem~\ref{thm:inflation_g_dimonoid}, the operations
$$
x\dashv y:=xy,\qquad
x\vdash y:=p_{r(x)r(y)}
$$
yield the $g$-dimonoid given by the following Cayley tables:
\[
\begin{array}{c|ccc}
\dashv & 0 & 1 & 2\\ \hline
0 & 0 & 0 & 0\\
1 & 0 & 0 & 0\\
2 & 0 & 0 & 2
\end{array}
\qquad\qquad
\begin{array}{c|ccc}
\vdash & 0 & 1 & 2\\ \hline
0 & 1 & 1 & 1\\
1 & 1 & 1 & 1\\
2 & 1 & 1 & 2
\end{array}
\]

Since $T^3=T$ and $p_0\neq0$, Theorem~\ref{thm:inflation_g_dimonoid} implies that this $g$-dimonoid is not a dimonoid. On the other hand, as $(S,\dashv)$ is commutative, the same theorem shows that it is iso-dual.

Consider also the canonical transversal $Q=\{0,2\}$, where $q_0=0$ and $q_2=2$. By Theorem~\ref{thm:inflation_g_dimonoid}, this transversal yields the trivial dimonoid $S$. The semigroup $(S,\cdot)$ has only the identity automorphism $\psi$ (see~Table~\ref{tab:auts3}), which clearly commutes with the retraction $r$. Therefore, we may apply Corollary~\ref{cor:isomorphic_transversals}.
However,
$$
\psi(p_0)=p_0=1\neq0=q_0=q_{\psi(0)}.
$$
Hence the condition of Corollary~\ref{cor:isomorphic_transversals} is not satisfied. Consequently, the $g$-dimonoids determined by the
transversals $P=\{1,2\}$ and $Q=\{0,2\}$ are not isomorphic.

\bigskip

The remaining three $g$-dimonoids are obtained by applying the constructions developed in Section~\ref{dcgdm}.

\medskip

\noindent{\bf Commutative $g$-dimonoid \#6.} Let $S=\{0,1,2\}$ be the semigroup given by the following Cayley table:
$$
\begin{array}{c|ccc}
\dashv & 0 & 1 & 2\\ \hline
0 & 0 & 0 & 0\\
1 & 0 & 0 & 0\\
2 & 0 & 0 & 1
\end{array}
$$
The semigroup $(S, \dashv)$ is monogenic of index 3 and period 1 generated by the element 2.

The multiplication of $S$ is of the form described in Proposition~\ref{comm_mon_gdm} with $a=2$, $b=1$, and $0_{\dashv}=0$.

Since $a,c,0_{\vdash}$ must be pairwise distinct, there are exactly two possibilities:
$$
(c,0_{\vdash})=(1,0)
\quad\text{or}\quad
(c,0_{\vdash})=(0,1).
$$

By Proposition~\ref{comm_mon_gdm}, the first choice yields the trivial commutative dimonoid $(S,\dashv,\dashv)$, whereas the second choice yields a $2$-dinilpotent commutative $g$-dimonoid, which fails to be a dimonoid by Proposition~\ref{comm_mon_gdm}, since $0_{\dashv}=0\neq 1 =0_{\vdash}$. The latter is given by the following Cayley tables:
$$
\begin{array}{c|ccc}
\dashv & 0 & 1 & 2\\ \hline
0 & 0 & 0 & 0\\
1 & 0 & 0 & 0\\
2 & 0 & 0 & 1
\end{array}
\qquad\qquad
\begin{array}{c|ccc}
\vdash & 0 & 1 & 2\\ \hline
0 & 1 & 1 & 1\\
1 & 1 & 1 & 1\\
2 & 1 & 1 & 0
\end{array}
$$

The above $g$-dimonoid is iso-dual by Corollary~\ref{isodual_comm_mon_gdm}, since
$$
b=1=0_{\vdash}
\quad\Longleftrightarrow\quad
c=0=0_{\dashv}.
$$

\bigskip

\noindent{\bf Commutative $g$-dimonoids \#7 and \#8.} Let $S=\{0,1,2\}$ be the null semigroup with zero $0$.
We construct commutative $g$-dimonoids using Proposition~\ref{comm_nullmon_gdm}. 

Since $a \neq 0_{\dashv}=0$, there are two possible choices for the distinguished element $a$, namely $a=2$ or $a=1$.

We first consider the case $a=2$.

Since $a,c,0_{\vdash}$ must be pairwise distinct, there are exactly two
possibilities:
$$
(c,0_{\vdash})=(1,0)
\quad\text{or}\quad
(c,0_{\vdash})=(0,1).
$$

The first choice yields the commutative dimonoid isomorphic to $O_3 \rbag \M_{3,1}$, whereas by Proposition~\ref{comm_nullmon_gdm} the second choice yields a $2$-dinil\-po\-tent commutative $g$-dimonoid given by the following Cayley tables:
$$
\begin{array}{c|ccc}
\dashv & 0 & 1 & 2\\ \hline
0 & 0 & 0 & 0\\
1 & 0 & 0 & 0\\
2 & 0 & 0 & 0
\end{array}
\qquad\qquad
\begin{array}{c|ccc}
\vdash & 0 & 1 & 2\\ \hline
0 & 1 & 1 & 1\\
1 & 1 & 1 & 1\\
2 & 1 & 1 & 0
\end{array}
$$

Its dual $g$-dimonoid has Cayley tables
$$
\begin{array}{c|ccc}
\dashv^d & 0 & 1 & 2\\ \hline
0 & 1 & 1 & 1\\
1 & 1 & 1 & 1\\
2 & 1 & 1 & 0
\end{array}
\qquad\qquad
\begin{array}{c|ccc}
\vdash^d & 0 & 1 & 2\\ \hline
0 & 0 & 0 & 0\\
1 & 0 & 0 & 0\\
2 & 0 & 0 & 0
\end{array}
$$

Thus, we obtain a pair of dual commutative $g$-dimonoids that are not isomorphic. Moreover, by Proposition~\ref{comm_nullmon_gdm}, neither of them is a dimonoid.

Now consider the second possibility $a=1$. Since $a,c,0_{\vdash}$ must be pairwise distinct, we obtain the commutative dimonoid isomorphic to $O_3 \rbag \M_{3,1}$ and the $2$-dinilpotent commutative $g$-dimonoid given by the following Cayley tables: 
$$
\begin{array}{c|ccc}
\dashv & 0 & 1 & 2\\ \hline
0 & 0 & 0 & 0\\
1 & 0 & 0 & 0\\
2 & 0 & 0 & 0
\end{array}
\qquad\qquad
\begin{array}{c|ccc}
\vdash & 0 & 1 & 2\\ \hline
0 & 2 & 2 & 2\\
1 & 2 & 0 & 2\\
2 & 2 & 2 & 2
\end{array}
$$

By Proposition~\ref{iso_comm_nullmon_gdm}, the latter $g$-dimonoid is isomorphic, via the transposition $\psi=(1\,2)$, to the commutative $g$-dimonoid, which is not a dimonoid, obtained above for $a=2$. Consequently, the two choices of $a$ determine the same isomorphism class of commutative $g$-dimonoids that are not dimonoids.

\end{document}